\documentclass{amsart}
\usepackage{amsmath, amssymb, latexsym}
\title{Optimal Betti numbers of forest ideals}
\author{Michael Goff}

\newtheorem{theorem}{Theorem}[section]

\newtheorem{conjecture}[theorem]{Conjecture}
\newtheorem{example}[theorem]{Example}

\newtheorem{question}[theorem]{Question}

\newcommand{\K}{\Gamma}
\newcommand{\field}{{\bf k}}
\newcommand{\codim}{\mbox{\upshape codim}\,}

\newcommand{\m}{\mu}

\newcommand{\lk}{\mbox{\upshape lk}\,}

\newcommand{\dist}{\mbox{\upshape dist}\,}

\newcommand{\LCM}{\mbox{\upshape lcm}\,}

\def\proof{\smallskip\noindent {\it Proof: \ }}

\def\endproof{\hfill$\square$\medskip}

\begin{document}

\begin{abstract}
We prove a tight lower bound on the Betti numbers of tree and forest ideals and a tight upper bound on certain graded Betti numbers of squarefree monomial ideals.
\end{abstract}

\date{August 31, 2008}

\maketitle

\section{Introduction}
In this paper we study bounds on Betti numbers of certain classes of edge ideals.  Several other papers, including \cite{LinSyz}, \cite{HVT}, and the survey paper \cite{EdgeIdealSurvey}, use combinatorial methods to describe the minimal free resolutions of edge ideals and bound their Betti numbers.  For example, Ferrers ideals, as described in \cite{Ferrers} and \cite{Ferrers2}, are conjectured in \cite{NR} and shown in \cite{nr2} to minimize Betti numbers among edge ideals of bipartite graphs.  Earlier papers construct bounds on Betti numbers in terms of the projective dimension \cite{BrunRomer} or the Hilbert function \cite{Bigatti}.  In general, the problem of characterizing the minimal free resolution of edges ideals is quite difficult and there are many open questions.  In particular, while constructing explicit (generally nonminimal) resolutions such as the Taylor resolution is effective in finding upper bounds on Betti numbers, there are no standard techniques for finding lower bounds.

We start by reviewing necessary background and introducing notation.  Throughout this paper $\field$ is an arbitrary field, and $S$ is the polynomial ring over $\field$ in variables $V = \{x_1,\ldots,x_n\}$ with the usual $\mathbb{Z}$-grading.  For a squarefree monomial ideal $I \subset S$, we consider the minimal free $\mathbb{Z}$-graded resolution:
\[ 0 \rightarrow
\bigoplus_{a \in \mathbb{Z}} S(-a)^{\beta_{l,a}} \rightarrow \ldots \rightarrow \bigoplus_{a \in \mathbb{Z}} 
S(-a)^{\beta_{0,a}} \rightarrow I \rightarrow 0. \]
In the above expression, $S(-a)$ denotes $S$ with grading shifted by $a$, and $l$ denotes the length of the resolution. In particular, $l \geq \codim(S/I)$.  The numbers $\beta_{i,a} = \beta_{i,a}(I)$ are called the the $\mathbb{Z}$-graded Betti numbers of $I$.  We also consider the ungraded Betti numbers $\beta_i = \beta_i(I) := \sum_{a \in \mathbb{Z}}\beta_{i,a}(I)$.

Squarefree monomial ideals are closely related to hypergraphs by the edge ideal construction.  A \textit{hypergraph} $G(V =\{x_1,\ldots,x_n\},E)$ is a vertex set $V$ and a set of edges $E \subset 2^V$ with the property that no edge is contained in another edge.  Note that we allow edges to have cardinality one, and we allow vertices that are not contained in an edge.  The \textit{degree} of $G$ is the maximum size of an edge.  A hypergraph is \textit{pure} if all its edges have the same cardinality.  The \textit{edge ideal} of $G$ is the ideal of $S$ given by $$I(G) := (x_{i_1}\ldots x_{i_r}: \{x_{i_1},\ldots,x_{i_r}\} \in E).$$  Since each squarefree monomial ideal $I$ has a unique set of minimal generators, there exists a unique hypergraph $G_I$ whose edge ideal is $I$.  Edge ideals were first introduced in \cite{EdgeIdeals}; results related to edge ideals can be found in \cite{Splittable}, \cite{Jacques}, \cite{Forests}, \cite{CharInd}, and \cite{EdgeIdeals}.

The outline of the paper is as follows.  We introduce notation and definitions in Section \ref{prelim}.  In Section \ref{ForestLB}, we prove a lower bound on the (ungraded) Betti numbers of hyperforest and hypertree ideals.  In Section \ref{UpperBounds}, we look at upper bounds on the graded Betti numbers of squarefree monomial ideals and prove a tight upper bound on $\beta_{2,3d-1}$ for degree $d$ ideals.

\section{Preliminaries}
\label{prelim}
One of our theorems is about hypertrees and hyperforests.  A \textit{hyperforest} is a hypergraph $G(V,E)$ with the property that the edges of $G$ can be enumerated $F_1,\ldots,F_{|E|}$ in such a way that for all $2 \leq i \leq |E|$, $|F_i-(F_1 \cup \ldots \cup F_{i-1})| \geq 1$.  If $G$ is pure and the edges of $G$ can be enumerated so that for all $2 \leq i \leq |E|$, $|F_i-(F_1 \cup \ldots \cup F_{i-1})| = 1$, then $G$ is called a \textit{hypertree}.  If $G$ is pure and has degree $2$, then hyperforests and hypertrees are ordinary graph-theoretic forests and trees.

We say that a hypergraph $G(V,E)$ is $k$-\textit{colorable} if there exists a function $\kappa:V \rightarrow [k]$, called a $k$-\textit{coloring}, such that no two vertices with the same $\kappa$-value belong to the same face.  All degree $d$ hyperforests are $d$-colorable.  Furthermore, all degree $d$ hypertrees have a unique $d$-coloring up to permutation of the colors.

We also use the notion of a simplicial complex.  A \textit{simplicial complex} $\K$ with the vertex set $V$ is a collection of subsets of $2^{V}$ called \textit{faces} such that $\K$ is closed under inclusion.  We do not insist that the singleton subsets of $V$ are faces.  With every simplicial complex $\K$ we associate its \textit{Stanley-Reisner ideal} $I_{\K} \subset S$ generated by non-faces of $\K$: $I_\K := (\prod_{x_i \in L}x_i: L \subset V, L \not\in \K)$ (see \cite{St96}).  Likewise, given a squarefree monomial ideal $I \subset S$, we denote by $\Gamma(I)$ the simplicial complex $\Gamma$ on $V$ whose Stanley-Reisner ideal is $I$.

If $W \subset V$, then the \textit{induced subcomplex of} $\K$ on $W$, denoted $\K[W]$, has vertex set $W$ and faces $\{F \in \K: F \subset W\}$.  If $v \in V$ and $\{v\}$ is a face in $\K$, then the \textit{link} of $v$, denoted $\lk_\K(v)$, is the simplicial complex that has vertex set $V-\{v\}$ and faces $\{G-\{v\}: v \in G \in \K\}$.  The \textit{antistar} of $v$ is $\K-v := \K[V-\{v\}]$.  Let $\tilde{\beta}_p(\K) := \dim_\field(\tilde{H}_p(\K; \field))$ be the dimension of the $p$-th reduced simplicial homology with coefficients in $\field$.  We make frequent use of Hochster's formula (see \cite[Theorem II.4.8]{St96}), which states that for $W \subset V$, $$\beta_{i,a}(I_\K) = \sum_{|W|=a}\tilde{\beta}_{i-|W|-2}(\K[W]).$$

One advantage of using simplicial complexes is that Mayer-Vietoris sequences, together with Hochster's formula, allow us to construct bounds on the Betti numbers of the corresponding squarefree ideal.

Simplicial complexes and hypergraphs can be related via the Stanley-Reisner ideal: with a simplicial complex $\K$, we associate a hypergraph $G_\K := G_{I_\K}$.  Likewise, with a hypergraph $G$, we associate the simplicial complex $\K(G) = \K(I(G))$.  Thus the edges of $G$ are the minimal nonfaces of $\K(G)$.  Also, $G_{\K(\tilde{G})} = \tilde{G}$ and $\K(G_{\tilde{\K}}) = \tilde{\K}$.

We can describe the operation of taking the link of a vertex on the level of hypergraphs.  If $v$ is a vertex of $G(V,E)$, then define $\lk_G(v) := G_{\lk_{\K(G)}(v)}$.  Equivalently, to construct $\lk_G(v)$, remove $v$ from $V$, and for all edges $F$ that contain $v$, replace $F$ by $F-\{v\}$; then delete any edges that become nonminimal under inclusion.  Similarly, we define the antistar $G-v$ on the level of hypergraphs by $G_{\K(G)-v}$.  We may construct $G-v$ by removing $v$ from $V$ and deleting all edges of $G$ that contain $v$.  We also define the \textit{induced hypergraph} on $W \subset V$ by $G[W]$; $G[W]$ has vertex set $W$ and edges $\{F: F \in E, F \subset W\}$.

We also use the Taylor resolution of a squarefree monomial ideal, which in general is not minimal.  Suppose $I$ is the edge ideal of the hypergraph $G(V,E)$ with $r$ edges.  For each $\{x_{j_1},\ldots,x_{j_t}\} = F_i \in E$, let $\m_i = x_{j_1} \ldots x_{j_t}$.  The Taylor resolution is a cellular resolution, in the sense of \cite{MillSt}, supported on the labeled simplex with $r$ vertices, labeled $\m_j$, $1 \leq j \leq r$.  For more information on cellular resolutions, see Chapter 4 of \cite{MillSt}.  In particular, the $\mathbb{Z}$-graded Betti numbers of the Taylor resolution are 
\begin{equation}
\label{TaylorBetti}
\beta_{(i-1),j}^T(I) = |\{W \subset [r]: |W| = i, \deg \LCM_{k \in W} \m_k = j \}|, \quad 1 \leq i \leq r.
\end{equation}

\noindent Here and throughout the paper, $[r] := \{1,2,\ldots,r\}$.

\section{Betti numbers of forest ideals}
\label{ForestLB}

Our first main theorem establishes a lower bound on the Betti numbers of tree ideals.

\begin{theorem}
\label{TreeLB}
Let $G$ be a degree $d$ hypertree on $n$ vertices, and suppose for $1 \leq i \leq d$, there are $n_i$ vertices of color $i$.  Let $I$ be the edge ideal of $G$.  Then for $j \geq 2$, $$\beta_{j-1}(I) \geq \sum_{i=1}^d {n_i \choose j}.$$
\end{theorem}
\proof We use induction on $n$.  In the case $n=d$, $G$ is a single edge and the result holds with $\beta_{j-1}(I) = 0$ for $j \geq 2$.  Consider $n>d$, and let $v$ be a leaf of $G$ (that is, a vertex contained in only one edge).  Every hypertree has a leaf.  Since $G-v$ is also a hypertree, the result holds on $I_{G-v}$ by the inductive hypothesis.

Suppose that $v$ is colored blue, and let $B \subset V$ be the set of blue vertices of $V$.  To prove the result, we show that for each $B' \subseteq B$, there exists $U' \subseteq V$ such that $U' \cap B = B'$, and $\tilde{H}_{|U'|-|B'|-1}(\K(G[U'])) \neq 0$.  The theorem then follows since, by taking all $B'$ with $|B'| = j$ and $v \in B'$, and using Hochster's formula, $\beta_{j-1}(I(G)) \geq \beta_{j-1}(I(G-v)) + {|B|-1 \choose j-1}$, which together with ${|B|-1 \choose j}+{|B|-1 \choose j-1} = {|B| \choose j}$ proves the result.

Fix $B'$ as above.  We describe an algorithm for constructing $U'$ with the desired property.  First initialize $\tilde{U} := V-(B-B')$, and $\tilde{G} := G[\tilde{U}]$.  Set $W := \emptyset$.  Initially, every blue vertex in $\tilde{G}$ is contained in an edge, and every edge contains a blue vertex.

\begin{itemize}

\item[Step 1:] Suppose there exists a non-blue $u \in \tilde{U}$ such that every blue vertex is contained in some face that does not contain $u$.  Then remove $u$ from $\tilde{U}$, and replace $\tilde{G}$ by $\tilde{G}-u$.  It is still the case that every edge contains a blue vertex, and every blue vertex is contained in an edge.  Repeat until there is no such vertex $u$.

\item[Step 2:] Choose $u$ to be a non-blue vertex.  There exists at least one blue $v \in \tilde{U}$ such that every edge that contains $v$ also contains $u$.  If $|W| = r$, then set $w_{r+1} := u$ and add it to $W$.  Remove $u$ from $\tilde{U}$ and replace $\tilde{G}$ by $\lk_{\tilde{G}}(u)$.  This process may create an edge that is a single blue vertex.  Also remove from $\tilde{U}$ and $\tilde{G}$ all vertices that are not contained in an edge after this process.  We argue below that no blue vertices are removed in this way.  Return to Step 1 if there are any non-blue vertices remaining in $\tilde{U}$.

\item[Step 3:] Take $U' = W \cup B'$.

\end{itemize}

We check that no blue vertex can be removed in Step 2.  Consider $\tilde{G}$ with the property that every blue vertex is contained in an edge, and every edge contains a blue vertex.  Consider the operation of taking the link of a non-blue vertex $u$.  Every edge in $\lk_{\tilde{G}}(u)$ contains a blue vertex.  Now consider a blue vertex $u_r$, which in $\tilde{G}$ is contained in the edge $F=\{u_1,\ldots,u_r\}$.  If $u \in F$, then $(F-\{u\})$ is an edge in $\lk_{\tilde{G}}(u)$ and $u_r$ is contained in an edge.  Otherwise, $F$ is removed from $\tilde{G}$ only if there exists an edge $F'$ with $u \in F'$ and $F'-\{u\} \subset F$.  In this case, since $F'-\{u\}$ contains a blue vertex, $u_r \in F'-\{u\}$.  Since $F'-\{u\}$ is an edge in $\lk_{\tilde{G}}(u)$, we conclude that $u_r$ is contained in an edge in $\lk_{\tilde{G}}(u)$.

Now we show that $\tilde{\beta}_{|U'|-|B'|-1}(\K(G[U'])) = 1$.  Let $\K_0 := \K(G[U'])$, and for $1 \leq i \leq |W|$, let $\K_i := \lk_{\K_{i-1}}(w_i)$.  Note that $\tilde{\beta}_{-1}(\K_{|W|})= 1$ since $\K_{|W|}$ has no non-empty faces.  The antistar $\K_{i}-w_{i+1}$ is a cone over a blue vertex by construction, and is therefore acyclic.  It follows from the portion of the Mayer-Vietoris sequence $$\tilde{H}_{|W|-i-1}(\K_{i}-w_{i+1}) \rightarrow \tilde{H}_{|W|-i-1}(\K_{i}) \rightarrow \tilde{H}_{|W|-i-2}(\K_{i+1}) \rightarrow \tilde{H}_{|W|-i-2}(\K_{i}-w_{i+1})$$ that $\tilde{\beta}_{|W|-i-1}(\K_{i}) = 1$.  We conclude that $\tilde{\beta}_{|U'|-|B'|-1}(\K(G[U'])) = 1$.
\endproof

\begin{example}
The lower bound in Theorem \ref{TreeLB} can be attained.
\end{example}
Label the vertex set of $G(V,E)$ by $\{v_1,\ldots,v_d,U_1,\ldots,U_d\}$, where for each $1 \leq i \leq d$, $U_i = \{u_{i,1},\ldots,u_{i,n_i-1}\}$.  Let $\{v_1,\ldots,v_d\} \in E$, and for all $1 \leq i \leq d$ and $1 \leq j \leq n_i-1$, let $\{v_1,\ldots,v_{i-1},v_{i+1},\ldots,v_d,u_{i,j}\} \in E$.  For each $1 \leq r \leq d$ and $W \subseteq V$, it can be verified that if $W \cap U_i = \emptyset$ for $i \neq r$, $v_i \in W$ for $i \neq r$, and $W$ is not simply $\{v_1\ldots,v_{r-1},v_{r+1},\ldots,v_d\}$, then $\tilde{\beta}_{d-2}(\K(G)[W]) = 1$ and $\tilde{\beta}_{p}(\K(G)[W]) = 0$ for $p \neq d-2$.  Furthermore, it can be verified that if $W$ does not satisfy these conditions, then $\K(G)[W]$ is acyclic.  It follows from Hochster's formula that $I(G)$ attains the lower bound of Theorem \ref{TreeLB}. \endproof

In the case of degree $2$ trees, we fully answer the question of equality.  For two vertices $u$ and $v$ of a connected graph $G$, let $\dist(u,v)$ be the number of edges in a shortest path joining $u$ and $v$.  The \textit{diameter} of $G$ is $\max_{u,v}(\dist(u,v))$.

\begin{theorem}
Let $G(V,E)$ be a colored tree with blue vertices $B$ and red vertices $R$, with $|B| = n_1$, and $|R| = n_2$.  Then $\beta_{j-1} = {n_1 \choose j} + {n_2 \choose j}$ for all $j \geq 2$ if and only if $G$ has diameter at most four.
\end{theorem}
\proof
First suppose $G$ has diameter greater than four.  Then $G$ has a subtree $G'$ that is a path on six vertices, of which three are red and three are blue.  One can check that $\beta_1(I(G')) = 7 > {3 \choose 2} + {3 \choose 2}$.  It follows by induction, as in the proof of Theorem \ref{TreeLB}, that $\beta_1(G) > {n_1 \choose 2} + {n_2 \choose 2}$.

Now suppose $G$ has diameter at most four.  There exists $v \in V$ such that for all $u \in V$, $\dist(u,v) \leq 2$.  Assume without loss of generality that $v$ is blue.  If $\dist(u,v)=1$, then $u$ is red, while if $\dist(u,v)=2$, then $u$ is blue.  Furthermore, all blue vertices except $v$ are leaves.  For each blue vertex $u \neq v$, let $p(u)$ be the unique neighbor of $u$.  Set $\K = \K(G)$.  We show that $\beta_{j-1} = {n_1 \choose j} + {n_2 \choose j}$ for all $j \geq 2$ by using Hochster's formula and considering induced subcomplexes on $W \subseteq V$ in several cases.

\begin{itemize}
 
\item[Case 1:] $v \not \in W$.  For all blue $u \in W$, $\K[W]$ is a cone over $u$ and therefore acyclic unless $p(u) \in W$.  Also, if $w \in W$ is a red vertex and $w \neq p(u)$ for any blue $u \in W$, then $\K[W]$ is a cone over $w$.  If $p(u) \in W$ for all blue $u \in W$, and all red vertices $w \in W$ satisfy $w = p(u)$ for some $u \in W$, then $\tilde{\beta}_{s-1}(\K[W]) = 1$ and $\tilde{\beta}_{t}(\K[W]) = 0$ for $t \neq s-1$, where $s = |W \cap R|$.  This is true for $s=1$ since in that case, $\K[W]$ is the disjoint union of a simplex and a vertex.  For $s>1$, $\K[W]$ is the join of $s$ such complexes, and the K\"{u}nneth formula applies.  Such subsets $W$ can be indexed by subsets of $B$ not containing $v$.

\item[Case 2:] $W \cap B = \{v\}$ and $W \cap R \neq \emptyset$.  Then $\tilde{\beta}_0(\K[W]) = 1$ and $\tilde{\beta}_j(\K[W]) = 0$ for $j \neq 0$, since in that case, $\K[W]$ is the disjoint union of a vertex and a simplex.  Such subsets $W$ can be indexed by nonempty subsets of $R$.

\item[Case 3:] $\{v\} \subsetneq W \cap B$.  Let $u \in W \cap B$ with $u \neq v$.  Then $\K[W]$ is a cone with apex $u$ unless $p(u) \in W$, so assume $p(u) \in W$.  Consider $\lk_\K(p(u))$.  By the reasoning of Case 1, $\lk_\K(p(u))[W]$ is acyclic unless for every $w \in R \cap W$, $w = p(z)$ for some $z \in W \cap B$, and also for all $z \in W \cap B$, $p(z) \in W$.  If that condition is also satisfied, then $\tilde{\beta}_{s-2}(\lk_{\K}(p(u))[W]) = 1$ and $\tilde{\beta}_{t}(\lk_{\K}(p(u))[W]) = 0$ for $t \neq s-2$, where $s = |W \cap R|$, by the reasoning of Case 1.  It then follows from the portion of the Mayer-Vietoris sequence $$\tilde{H}_t(\K[W]-p(u)) \rightarrow \tilde{H}_t(\K[W]) \rightarrow \tilde{H}_{t-1}(\lk_{\K}(p(u))[W]) \rightarrow \tilde{H}_{t-1}(\K[W]-p(u))$$ that $\tilde{\beta}_{t}(\K[W]) = \tilde{\beta}_{t-1}(\lk_\K(p(u))[W])$ for all $t$.  Such subsets $W$ with $\K[W]$ not acyclic can be indexed by subsets of $B$ containing $v$.

\end{itemize}

It follows from Hochster's formula and the three cases above that $\beta_{j-1} = {n_1 \choose j} + {n_2 \choose j}$ for all $j \geq 2$.
\endproof

\begin{question}
For degree $d$ hypertree ideals, when is $\beta_{j-1}(I) = \sum_{i=1}^d {n_i \choose j}$ for all $j \geq 2$?
\end{question}

Theorem \ref{TreeLB} can be used to establish a lower bound on the Betti numbers of forest ideals.  If $r$ is an integer, we say that the sequence of integers $(r_1,\ldots,r_d)$ is a \textit{nearly even d-partition} of $r$ if for all $1 \leq i < j \leq d$, $|r_i-r_j| \leq 1$.

\begin{theorem}
\label{ForestLowerBound}
Let $T$ be a degree $d$ hyperforest with $t$ edges, and let $(n_1,\ldots,n_d)$ be a nearly even $d$-partition of $t+d-1$.  Let $I$ be the edge ideal of $T$.  Then for $j \geq 2$, $$\beta_{j-1}(I) \geq \sum_{i=1}^d {n_i \choose j}.$$
\end{theorem}
\proof
Enumerate the edges of $T$ by $F_1,\ldots,F_{t}$ such that for $2 \leq i \leq t$, $|F_i-\{F_1 \cup \ldots \cup F_{i-1}\}| \geq 1$.  For $2 \leq k \leq t$, choose $v_k \in F_k-(F_1 \cup \ldots \cup F_{k-1})$.  Let $\kappa$ be a coloring of $T$, and set $n'_i = 1+|\{k: 2 \leq k \leq t, \kappa(v_k)=i\}|$.  It then follows by the same argument of the proof of Theorem \ref{TreeLB} that for $j \geq 2$, $\beta_{j-1}(I) \geq \sum_{i=1}^d {n'_i \choose j}.$

To complete the proof, observe that among all partitions $\{n'_1,\ldots,n'_d\}$ of $t+d-1$, the sum $\sum_{i=1}^d {n'_i \choose j}$ is minimized when $\{n'_1,\ldots,n'_d\}$ is nearly even.
\endproof

\section{Upper Bounds on Graded Betti Numbers}
\label{UpperBounds}
In this section we establish some upper bounds on the graded Betti numbers of squarefree monomial ideals.

A simple observation is that for a squarefree monomial ideal $I$ with $t$ generators, $\beta_{i-1,j}(I) \leq {t \choose i}$.  This follows from the Taylor resolution.  For pure degree $d$ ideals and $j = d+(i-1)r$, $\beta_{i-1,j}(I) = {t \choose i}$ if the variables are labeled $x_1, \ldots, x_{d-r}, x_{a,1}, \ldots, x_{a,r}$ for $1 \leq a \leq t$ and $I$ consists of $t$ generators of the form $\mu_a = x_1\ldots x_{d-r} x_{a,1}\ldots x_{a,r}$.

\begin{theorem}
\label{beta35}
Let $I$ be a pure degree $d$ squarefree monomial ideal with $t$ generators, and let $(t_1,t_2)$ be a nearly even $2$-partition of $t$.  Then $\beta_{2,3d-1}(I) \leq {t \choose 3} - {t_1 \choose 3} - {t_2 \choose 3}$.
\end{theorem}
\begin{proof} We show, using the Taylor resolution, that actually $$\beta^T_{2,3d-1}(I) \leq {t \choose 3} - {t_1 \choose 3} - {t_2 \choose 3}.$$  Suppose that $I$ is the edge ideal of a hypergraph $G(V,E)$ with edges $\{F_1,\ldots,F_t\}$.  Construct a graph $G'$ with vertex set $\{v_1,\ldots,v_t\}$ so that $(v_i,v_j)$ is an edge in $G'$ if and only if $F_i \cap F_j \neq \emptyset$.

We calculate $$\beta^T_{2,3d-1}(I) = \frac{1}{2}|\{(F_i,F_j,F_k) \in E^3: F_i \cap F_j = \emptyset, F_i \cap F_k = \emptyset, |F_j \cap F_k| = 1\}| \leq$$ $$\frac{1}{2}|\{(v_i,v_j,v_k) \in V(G')^3: (v_i,v_j) \not\in E(G'), (v_i,v_k) \not\in E(G'), (v_j,v_k) \in E(G')\}|.$$

Label the latter expression by $P(G')$; $P(G')$ is the number of induced copies of a single edge on $3$ vertices in $G'$.  With $\deg v$ denoting the degree of a vertex $v$ in $G'$, and $a$ the average of $\deg v$ over all vertices $v$ in $G'$, we have $$\beta^T_{2,3d-1}(I) \leq \frac{1}{2}|\{(v_i,v_j,v_k) \in V(G')^3: (v_i,v_k) \not\in E(G'), (v_j,v_k) \in E(G')\}| = $$ $$\frac{1}{2}\sum_{v \in V(G')} (\deg v)(t-1-\deg v) \leq \frac{1}{2}ta(t-1-a).$$

We apply induction on $t$ with the base cases $P(G') = 0$ for $t=1,2$ clear.  First consider the case $t=2k$ is even.  If $a \leq k-1$, then $P(G') \leq {t \choose 3} - {t_1 \choose 3} - {t_2 \choose 3}$.  Otherwise, if $a > k-1$, then there exists a vertex $v \in V(G')$ with $\deg v \geq k$.  By the inductive hypothesis, $P(G'-v) \leq {t-1 \choose 3} - {k \choose 3} - {k-1 \choose 3}$.  Also $P(G') - P(G'-v) \leq {t-1 \choose 2} - {k \choose 2}$ since the induced subgraph on $v$ and vertices $u_1,u_2$ is not a single edge if $u_1$ and $u_2$ are both neighbors of $v$.  Hence the desired inequality holds on $P(G')$.

Now consider $t = 2k+1$.  If $a \leq k-1$, then $P(G') < {t \choose 3} - {t_1 \choose 3} - {t_2 \choose 3}$.  If $a > k-1$, then there exists $v \in V(G')$ with $\deg v \geq k$.  By the inductive hypothesis, $P(G'-v) \leq {t-1 \choose 3} - {k \choose 3} - {k \choose 3}$.  Also $P(G') - P(G'-v) \leq {t-1 \choose 2} - {k \choose 2}$.  Hence the desired inequality holds on $P(G')$.
\end{proof}

The upper bound of Theorem \ref{beta35} is attained by the degree $2$ hypergraph with vertices \newline $u_1,u_2,v_1, \ldots, v_{t_1}, w_1,\ldots,w_{t_2}$ and edges $$\{(u_1,v_1)\ldots(u_1,v_{t_1}),(u_2,w_1)\ldots(u_2,w_{t_2})\}.$$

\noindent What about other Betti numbers of the form $\beta_{2,\bullet}(I)$?  A higher bound is necessary for $\beta_{2,6}$ in the case $d=3$.  There exists a degree $3$ squarefree monomial ideal on $6$ generators with $\beta_{2,6} = {6 \choose 3}$, namely $$I=\{x_1x_2x_4, x_1x_2x_5, x_1x_3x_6, x_1x_3x_7, x_2x_3x_8, x_2x_3x_9\}.$$ It can also be seen by exhaustive search that $I$ is the only ideal, up to isomorphism, satisfying these properties, and that there does not exist a degree $3$ ideal on $7$ generators satisfying $\beta_{2,6} = {7 \choose 3}$ or even $\beta^T_{2,6} = {7 \choose 3}$; that is, there is no degree $3$ hypergraph with $7$ edges such that all sets of $3$ edges have a union with $6$ vertices.  It follows that, for a degree $3$ squarefree monomial ideal $I$ on $t$ generators, $\beta_{2,6}(I) \leq {t \choose 3} - T(t,7,3)$.  Here $T(t,7,3)$ is the Tur\'{a}n number, defined to be the minimum cardinality of a collection $W$ of $3$-subsets of $[t]$ such that every $7$-subset of $[t]$ contains an element of $W$.  See the reference paper  \cite{Turan} for more on Tur\'{a}n numbers.

In analogy with Theorem \ref{beta35}, we make the following conjecture.

\begin{conjecture}
\label{b36}
Let $I$ be a degree $3$ squarefree monomial ideal with $t$ generators, and let $(t_1,t_2,t_3)$ be a nearly even $3$-partition of $t$.  Then $\beta_{2,6}(I) \leq {t \choose 3}-{t_1 \choose 3}-{t_2 \choose 3}-{t_3 \choose 3}$.
\end{conjecture}

\noindent If Conjecture \ref{b36} is true, then the bound is tight and is attained by the ideal on variables $x_1,x_2,x_3, y_1,\ldots,y_{t_1}, z_1,\ldots,z_{t_2}, w_1,\ldots,w_{t_3}$ given by $$I = (x_1x_2y_1,\ldots,x_1x_2y_{t_1},x_1x_3z_1,\ldots,x_1x_3z_{t_2},x_2x_3w_1,\ldots,x_2x_3w_{t_3}).$$

\section*{Acknowledgments}
The author has been supported while working on this project by a graduate fellowship from VIGRE NSF Grant DMS-0354131.

\end{document}